\documentclass[a4paper,12pt]{article}
\usepackage[left=2cm,top=2cm,right=2cm,bottom=2cm]{geometry}
\usepackage{amsmath}
\usepackage{amssymb}
\usepackage{amsfonts}
\usepackage{amsthm}

\DeclareMathOperator{\vol}{vol}
\DeclareMathOperator{\lin}{lin}
\DeclareMathOperator{\interior}{int}

\DeclareMathOperator{\GL}{GL}

\newtheorem{thm}{Theorem}[section]
\newtheorem{lemma}[thm]{Lemma}

\newtheorem{conj}[thm]{Conjecture}
\newtheorem{prob}[thm]{Problem}

\theoremstyle{definition}

\theoremstyle{definition}

\newcommand{\bears}{\begin{eqnarray*}}
\newcommand{\eears}{\end{eqnarray*}}

\newcommand{\ZZ}{\ensuremath{\mathbb{Z}}}

\newcommand{\RR}{\ensuremath{\mathbb{R}}}

\newcommand{\La}{\ensuremath{\Lambda}}
\newcommand{\la}{\ensuremath{\lambda}}

\newcommand{\DK}{\ensuremath{\mathfrak{D}K}}
\newcommand{\hDK}{\ensuremath{\tfrac{1}{2}\mathfrak{D}K}}

\newcommand{\sm}{\ensuremath{\setminus}}
\newcommand{\ssq}{\ensuremath{\subseteq}}

\newcommand{\vn}{\ensuremath{\varnothing}}

\newcommand{\bsm}{\ensuremath{\boldsymbol}}

\numberwithin{equation}{section}

\title{Lattice-point enumerators of ellipsoids}
\author{Romanos-Diogenes Malikiosis}
\date{}

\begin{document}
\maketitle

\begin{abstract}
Minkowski's second theorem on successive minima asserts that the volume of a $0$-symmetric convex body $K$ over the covolume of
a lattice $\La$ can be bounded above by a quantity involving all the successive minima of $K$ with respect to $\La$. We will
prove here that the number of lattice points inside $K$ can also accept an upper bound of roughly the same size, in the special
case where $K$ is an ellipsoid. Whether this is also true for all $K$ unconditionally is an open problem, but there is
reasonable hope that the inductive approach used for ellipsoids could be extended to all cases.
\end{abstract}

\section{Introduction}

In 1993, Betke, Henk, and Wills \cite{BHW} conjectured that the lattice-point enumerator of $\bsm{0}$-symmetric convex bodies accepts a similar
upper bound as its volume, which is described by Minkowski's second theorem on successive minima, namely 
\begin{equation}\label{BHW}
 \lvert K\cap\La\rvert\leqslant \prod_{i=1}^{d}\left[\frac{2}{\la_i(K,\La)}+1\right]
\end{equation}
where $K\ssq\RR^d$ is a convex body, symmetric about the origin, $\La$ a lattice, $\lvert T\rvert$ the cardinality of a set $T$, and $\la_i(K,\La)$ the 
$i$th successive minimum of $K$ with respect to $\La$; it is defined as follows
\[\la_i(K,\La)=\inf\{\la>0|(\la K)\cap\La\text{ contains }i\text{ linearly independent points}\}.\]
Here, a convex body is a compact convex subset of $\RR^d$; we don't need any additional assumptions, such as 
$\bsm{0}\in\interior(K)$. We also denote by $q_i(K,\La)$ or just $q_i$ for short, the quantities 
\[\left[\frac{2}{\la_i(K,\La)}+1\right]\]
 that appear on the right-hand side of \eqref{BHW}. We remind Minkowski's second theorem on successive minima:
\[\frac{\vol(K)}{\det(\La)}\leqslant\prod_{i=1}^{d}\frac{2}{\la_i(K,\La)}\]
where $\det(\La)$ is the determinant (or the covolume) of the lattice $\La$, which is equal to the volume of any fundamental
parallelepiped of $\La$. Inequality \eqref{BHW} implies Minkowski's second theorem on successive minima \cite{BHW}. Even though the 
lattice-point enumerator approximates the volume of the convex body over the covolume of the lattice, such an upper bound is 
not easy to derive. The main obstruction is the lack of homogeneity for the lattice-point enumerator, a property enjoyed by 
volumes.

Betke, Henk, and Wills  proved such an inequality extending Minkowski's first theorem on successive minima \cite{BHW},
in particular, $\lvert K\cap\La\rvert\leqslant q_1^d$; as for inequality \eqref{BHW}, they proved it to be true in two 
dimensions (the one-dimensional being trivial). Several attempts for weaker inequalities followed; Henk proved that inequality
\eqref{BHW} is true up to multiplying the right-hand side by $2^{d-1}$ \cite{H}. The author then managed to decrease this exponential factor
to roughly $1.64^d$ \cite{M10}, which yields the best inequality known unconditionally.

Regarding families of convex bodies, Bey, Henk, Henze, and Linke, confirmed the conjecture for lattice parallelepipeds and lattice-face
polytopes \cite{BHHL}. In this note, we will confirm the conjecture for all ellipsoids, without requiring to be centered at the 
origin:
\begin{thm}\label{mainthm}
 Let $E\ssq\RR^d$ be an ellipsoid and $\La$ a lattice. Then
\[\lvert E\cap\La\rvert\leqslant\prod_{i=1}^{d}\left[\frac{1}{\la_i(\mathfrak{D}E,\La)}+1\right].\]
\end{thm}
$\mathfrak{D}T$ denotes the difference set of $T\ssq\RR^d$, that is
\[\mathfrak{D}T:=T-T=\{x-y|x,y\in T\}.\]
We extended the notion of the successive minima to convex bodies that are not necessarily $\bsm{0}$-symmetric, using the
following symmetrization
\[\la_i(K,\La):=\la_i(\tfrac{1}{2}\mathfrak{D}K,\La)=2\la_i(\mathfrak{D}K,\La)\]
and we also extend the definition of the quantities $q_i(K,\La)$ as well, using the same symmetrization. Under this extension,
Minkowski's two theorems on successive minima still hold, due to the Brunn-Minkowski inequality that yields
\[\vol(K)\leqslant \vol(\tfrac{1}{2}\mathfrak{D}K)\]
for all convex bodies $K$. Furthermore, the author proved in \cite{M11} that 
\[\lvert K\cap\La\rvert\leqslant q_1^d\]
 holds for all $K$, whether $\bsm{0}$-symmetric or not, and in general
\[\lvert K\cap\La\rvert\leqslant\frac{4}{e}(\sqrt{3})^{d-1}\prod_{i=1}^{d}q_i.\]

The argument for ellipsoids follows the line of ideas in \cite{M11}. In that paper, the author managed to reduce 
inequality \eqref{BHW}, to the following simultaneous translation problem:
\begin{prob}\label{STP}
 Let $K_1,K_2,\dotsc,K_n$ be $d$-dimensional convex bodies and $\La$ a lattice such that for all indices $i\neq j$
\[(K_i-K_j)\cap\La=\vn.\]
Prove that for each $t\geqslant1$ there are lattice vectors $\bsm{v}_{1,t},\dotsc,\bsm{v}_{n,t}$, such that the translated
convex bodies $K'_i=K_i+\bsm{v}_{i,t}$ satisfy
\[(K'_i-K'_j)\cap t\La=\vn.\]
\end{prob}
It should be noted that a weaker version of this problem implies the desired conjecture. In the case of spheres or homothetic
ellipsoids, we will prove something stronger, namely that we can pick translation vectors $\bsm{v}_i=\bsm{v}_{i,t}$ that
satisfy the second condition for all $t\geqslant1$ (Theorem \ref{STPS}). 






\section{Outline of the proof}

In \cite{M11}, the author proposed the following more general conjecture:

\begin{conj}\label{strconj}
Let $K_1,\dotsc,K_n\subset \mathbb{R}^d$ be convex bodies and $\La$ a lattice. Also, let $\mathbf{e}_1,\dotsc,\mathbf{e}_d$ be a basis of $\La$ and denote by 
$\La^i$ the $\ZZ$-span of $0,\mathbf{e}_1,\dotsc,\mathbf{e}_i$, and let $q_1\geqslant q_2\geqslant\dotsb\geqslant q_d\geqslant q_{d+1}$ be positive 
integers satisfying
\vspace{-.25in}
\begin{description}
	\item[(C1)] $\mathfrak{D}K_j\cap q_i(\Lambda\setminus\Lambda^{i-1})=\varnothing$ for all $1\leqslant j\leqslant n$ and $1\leqslant i\leqslant d$.
	\item[(C2)] $(K_j-K_l)\cap q_{d+1}\Lambda = \varnothing$ for all $1\leqslant j,l\leqslant n$, $j\neq l$. 
\end{description}
\vspace{-.25in}
Then
\[\sum_{j=1}^n G(K_j,\Lambda)\leqslant \prod_{i=1}^d q_i.\]
\end{conj}
The main reason for introducing a stronger conjecture is the possibility of using induction on the dimension. A possible ``proof''
would consist of the following steps:
\vspace{-.25in}
\begin{description}
 \item [(1)] Translate the $K_i$ by lattice vectors, so that $q_{d+1}$ is replaced by $q_d$ in \textbf{(C2)}. Notice that neither
\textbf{(C1)} nor the total lattice point enumerator is changed when we apply these translations.
 \item [(2)] For fixed integer $r$, consider all the intersections of $K_i-t\mathbf{e}_d$ by the hyperplane $V^{d-1}:=\La^{d-1}\otimes_{\ZZ}\RR$,
for all $t\equiv r\pmod{q_d}$ and all $1\leqslant i\leqslant n$. Denote those intersections by $K_{i,t}$, and verify that they
satisfy conditions \textbf{(C1)} and \textbf{(C2)} for the lattice $\La^{d-1}$, the basis $\mathbf{e}_1,\dotsc,\mathbf{e}_{d-1}$, and integers 
$q_1\geqslant q_2\geqslant \dotsb \geqslant q_d$.
 \item [(3)] Apply induction and verify that the total lattice point enumerator satisfies the desired inequality.
\end{description}
The main problem is with step \textbf{(1)}, and this is how we prove that Problem \ref{STP} is a reduction of inequality
\eqref{BHW}. What we will show in this paper, is that the above procedure works in the case where all the $K_i$ are spheres or
homothetic ellipsoids, thus proving Theorem \ref{mainthm}. We note that the family of spheres in all
dimensions is closed under intersections by affine subspaces.

Problem \ref{STP} for spheres is tackled as follows: fix $K_1$, and translate the other spheres by vectors of $\La$, so that
they get as close as possible to $K_1$. This way, $K_1-K_j$ avoids $t\La$, for all $t\geqslant 1$, $j\geqslant1$. Now, for 
$2\leqslant i<j\leqslant n$, $K_i$ and $K_j$ might not be as close as possible, but still they are close enough, so that 
$K_i-K_j$ still avoids $t\La$, for all $t\geqslant 1$. For the last part we use the parallelogram law, a property enjoyed only 
by Euclidean norms.

Perhaps a similar idea might work for the general case; the major obstruction is that when the $K_i$ are arbitrary, the difference
bodies $K_i-K_j$ give rise to totally different norms, and we cannot use a nice property, such as the parallelogram law. Heuristically,
we expect that such translations should always exist; what \textbf{(C2)} asserts is that $K_i$ avoids $K_j$, even if it is translated
by $q_{d+1}\La$. Normally, this would also be possible if we replace $q_{d+1}\La$ by the sparser lattice $q_d\La$, perhaps after
we translate the $K_i$ by vectors of $\La$.

The presentation of the results herein will
be as self-contained as possible; the reader need not refer to the results of \cite{M10} or \cite{M11} in order to understand
this paper.

\section{Simultaneous translation of spheres}

Here we answer Problem \ref{STP} to the affirmative, in the case where all the $K_i$ are spheres; by $\mathbf{B}(\bsm{w},r)$
we denote the sphere with radius $r$, centered at $\bsm{w}$.
\begin{thm}\label{STPS}
Let $K_i=\mathbf{B}(\bsm{w}_i,r_i)$, where $1\leqslant i\leqslant n$, $\bsm{w}_i\in\RR^d$, $r_i>0$ for all $i$. Also, let $\La$ a lattice such that
\[(K_i-K_j)\cap\La=\vn\]
for all $1\leqslant i<j\leqslant n$. Then there are $\bsm{u}_i\in \bsm{w}_i+\La$ for all $i$, such that for all $t\geqslant1$ we have
\[(K'_i-K'_j)\cap t\La=\vn\]
for all $1\leqslant i<j\leqslant n$, where $K'_i=\mathbf{B}(\bsm{u}_i,r_i)$.
\end{thm}

\begin{proof}
We define
\[d_{ij}:=\inf\{\left\|\bsm{w}_i-\bsm{w}_j+\bsm{\la}\right\||\bsm{\la}\in\La\}\]
for $1\leqslant i<j\leqslant n$, where $\left\|\cdot\right\|$ denotes the Euclidean norm. The above can also be viewed as the Euclidean distance between the sets $\bsm{w}_i+\La$ and $\bsm{w}_j+\La$.

Now we put $\bsm{u}_1=\bsm{w}_1$ and for $2\leqslant i\leqslant n$ we choose $\bsm{u}_i\in \bsm{w}_i+\La$ such that $\left\|\bsm{u}_1-\bsm{u}_i\right\|=d_{1i}$. We define
\[d^t_{ij}:=\inf\{\left\|\bsm{u}_i-\bsm{u}_j+t\bsm{\la}\right\||\bsm{\la}\in\La\}.\]
We will prove that $d^t_{ij}\geqslant d_{ij}$ for all $t\geqslant1$ and $1\leqslant i<j\leqslant n$. For $i=1$ and all $\bsm{\la}\in\La$ we have
\[\left\|\bsm{u}_1-\bsm{u}_j+\bsm{\la}\right\|\geqslant\left\|\bsm{u}_1-\bsm{u}_j\right\|=d_{1j}\]
which is equivalent to
\begin{equation}\label{dis}
2\left\langle \bsm{u}_1-\bsm{u}_j,\bsm{\la}\right\rangle+\left\|\bsm{\la}\right\|^2\geqslant0,
\end{equation}
taking squares on both sides, where $\left\langle \cdot,\cdot\right\rangle$ is the usual Euclidean inner product. Hence,
\begin{eqnarray*}
\left\|\bsm{u}_1-\bsm{u}_j+t\bsm{\la}\right\|^2 &=& \left\|\bsm{u}_1-\bsm{u}_j\right\|^2+2t\left\langle \bsm{u}_1-\bsm{u}_j,\bsm{\la}\right\rangle+t^2\left\|\bsm{\la}\right\|^2\\
&=& (d_{1j})^2+t(2\left\langle \bsm{u}_1-\bsm{u}_j,\bsm{\la}\right\rangle+\left\|\bsm{\la}\right\|^2)+t(t-1)\left\|\bsm{\la}\right\|^2\\
&\geqslant& (d_{1j})^2
\end{eqnarray*}
by \eqref{dis} and $t\geqslant1$. Since this holds for all $\la\in\La$, we get $d^t_{1j}\geqslant d_{1j}$ for all $t$ and $j$.

Next, assume that $2\leqslant i<j\leqslant n$. We will need the following:

\begin{lemma}
Let $\bsm{u}\in\RR^d$ such that
\[\left\|\tfrac{1}{2}\bsm{u}+\bsm{\la}\right\|\geqslant\left\|\tfrac{1}{2}\bsm{u}\right\|\]
for all $\bsm{\la}\in\La$. Then for all $t\geqslant1$ and $\bsm{\la}\in\La$ the following inequality holds
\[\left\|\bsm{u}+\bsm{\la}\right\|\leqslant\left\|\bsm{u}+t\bsm{\la}\right\|.\]
\end{lemma}

\begin{proof}[Proof of Lemma]
By squaring both sides of the first inequality we obtain the following equivalent inequality:
\begin{equation}\label{lm}
\left\langle \bsm{u},\bsm{\la}\right\rangle+\left\|\bsm{\la}\right\|^2\geqslant0.
\end{equation}
Similarly, the second inequality is equivalent to
\[2\left\langle \bsm{u},\bsm{\la}\right\rangle+\left\|\bsm{\la}\right\|^2\leqslant 2t\left\langle \bsm{u},\bsm{\la}\right\rangle+t^2\left\|\bsm{\la}\right\|^2\]
which in turn is equivalent to
\[(t^2-1)\left\|\bsm{\la}\right\|^2+2(t-1)\left\langle \bsm{u},\bsm{\la}\right\rangle\geqslant0,\]
and since $t\geqslant1$, the above is equivalent to
\[(t+1)\left\|\bsm{\la}\right\|^2+2\left\langle \bsm{u},\bsm{\la}\right\rangle\geqslant0,\]
which is true from \eqref{lm} and again from $t\geqslant1$.
\end{proof}

It suffices to prove that $\bsm{u}_i-\bsm{u}_j$ satisfies the conditions of the Lemma. Indeed, by the parallelogram law,
\begin{eqnarray*}
\left\|\tfrac{1}{2}(\bsm{u}_i-\bsm{u}_j+2\bsm{\la})\right\| &=& \tfrac{1}{2}\left\|\bsm{u}_i-\bsm{u}_1+\bsm{\la}\right\|^2+\tfrac{1}{2}\left\|\bsm{u}_1-\bsm{u}_j+\bsm{\la}\right\|^2\\
&\phantom{=}& -\left\|\tfrac{1}{2}(\bsm{u}_i+\bsm{u}_j-2\bsm{u}_1)\right\|^2\\
&\geqslant& \tfrac{1}{2}(d_{1i})^2+\tfrac{1}{2}(d_{1j})^2-\left\|\tfrac{1}{2}(\bsm{u}_i+\bsm{u}_j-2\bsm{u}_1)\right\|^2\\
&=& \tfrac{1}{2}(d_{1i})^2+\tfrac{1}{2}(d_{1j})^2+\left\|\tfrac{1}{2}(\bsm{u}_i-\bsm{u}_j)\right\|^2\\
&\phantom{=}& -\tfrac{1}{2}\left\|\bsm{u}_i-\bsm{u}_1\right\|^2-\tfrac{1}{2}\left\|\bsm{u}_1-\bsm{u}_j\right\|^2\\
&=& \left\|\tfrac{1}{2}(\bsm{u}_i-\bsm{u}_j)\right\|^2.
\end{eqnarray*}
Here, we used the parallelogram law for the parallelogram with sides $\bsm{u}_i-\bsm{u}_1+\bsm{\la}$ and $\bsm{u}_1-\bsm{u}_j+\bsm{\la}$ and the parallelogram with sides $\bsm{u}_i-\bsm{u}_1$ and $\bsm{u}_1-\bsm{u}_j$. The Lemma clearly shows that $d^t_{ij}\geqslant d_{ij}$ for all $t\geqslant1$. By hypothesis, we have $d_{ij}>r_i+r_j$ for all $i\neq j$, therefore $d^t_{ij}>r_i+r_j$ must hold for all $i\neq j$, yielding the fact that
\[(K'_i-K'_j)\cap t\La=\vn\]
for all $t\geqslant1$ and $i\neq j$, as desired.
\end{proof}

\section{Proof of Conjecture \ref{strconj} for spheres}

In \cite{M11}, this conjecture was proven for all convex bodies for $d=1$ and
$d=2$:
\begin{thm}\label{strong}
 Let $S_1,\dotsc,S_n\ssq\RR^d$ be spheres and $\La$ a lattice. Also, let $\mathbf{e}_1,\dotsc,\mathbf{e}_d$ be a basis of $\La$ and denote by 
$\La^i$ the $\ZZ$-span of $0,\mathbf{e}_1,\dotsc,\mathbf{e}_i$, and let $q_1\geqslant q_2\geqslant\dotsb\geqslant q_d\geqslant q_{d+1}$ be positive 
integers satisfying
\vspace{-.25in}
\begin{description}
 \item [(C1)] $\mathfrak{D}S_j\cap q_i(\La\sm\La^{i-1})=\vn$ for all $1\leqslant j\leqslant n$ and $1\leqslant i\leqslant d$.
 \item [(C2)] $(S_j-S_k)\cap q_{d+1}\La=\vn$ for all $1\leqslant j,k\leqslant n$, $j\neq k$.
\end{description}
\vspace{-.25in}
Then
\[\sum_{j=1}^{n}\lvert S_j\cap\La\rvert\leqslant\prod_{i=1}^{d}q_i.\]
\end{thm}

\begin{proof}
 We will use induction on $d$; for $d=1$ this is already proven, as mentioned above. Assume that it holds for $d-1$. For each
sphere $S_i$ and each $m\in\RR$ define the ``slice'' of height $m$ by
\[S_j[m]=\{\bsm{x}\in S_j|\bsm{x}-m\mathbf{e}_d\in V^{d-1}\}\]
where $V^i=\La^i\otimes_{\ZZ}\RR$ the vector subspace spanned by $0,\mathbf{e}_1,\dotsc,\mathbf{e}_i$. Also, denote
\[S_{j,m}=S_j[m]-m\mathbf{e}_d\]
the projection of $S_j[m]$ on $V^{d-1}$ along $\mathbf{e}_d$. Hence,
\begin{eqnarray*}
 \sum_{j=1}^{n}\lvert S_j\cap\La\rvert &=& \sum_{m\in\ZZ}\sum_{j=1}^{n}\lvert S_{j,m}\cap\La^{d-1}\rvert\\
&=& \sum_{r=1}^{q_d}\sum_{m\equiv r\bmod{q_d}}\sum_{j=1}^{n}\lvert S_{j,m}\cap\La^{d-1}\rvert.
\end{eqnarray*}
It suffices to prove that for each $r$ we have
\begin{equation}\label{lowerdim}
\sum_{m\equiv r\bmod{q_d}}\sum_{j=1}^{n}\lvert S_{j,m}\cap\La^{d-1}\rvert\leqslant\prod_{i=1}^{d-1}q_i.
\end{equation}
Thus, by induction, we only need to verify conditions (\textbf{C1}) and (\textbf{C2}) for the $(d-1)$-dimensional spheres
$S_{j,m}$, $\mathbf{e}_1,\dotsc,\mathbf{e}_{d-1}$ and the integers $q_1,\dotsc,q_d$, where $m\equiv r\pmod{q_d}$ for some fixed $r$. For each
$j$ and $i$ with $1\leqslant i\leqslant d-1$ we have
\[\mathfrak{D}S_{j,m}\cap q_i(\La^{d-1}\sm\La^{i-1})\ssq\mathfrak{D}S_j\cap q_i(\La\sm\La^{i-1})=\vn\]
by hypothesis, so (\textbf{C1}) is satisfied. (\textbf{C2}) is not necessarily satisfied; however, as we shall see, it is 
satisfied for some appropriate translations of these spheres.

We apply Theorem \ref{STPS} for the spheres $S_1,\dotsc,S_n$ and the lattice $q_{d+1}\La$. So, we can translate these
spheres by elements of $q_{d+1}\La$ so that instead of (\textbf{C2}) they satisfy the stronger condition
\begin{equation}\label{C2strong}
 (S_j-S_k)\cap t\La=\vn
\end{equation}
for all $1\leqslant j,k\leqslant n$, $j\neq k$ and all $t\geqslant q_{d+1}$ (we denote the translated spheres again by $S_j$,
in order to keep the notation compact). In particular, it holds for $t=q_d$. It should be emphasized that under translation
by $q_{d+1}\La$, the total lattice point enumerator
\[\sum_{j=1}^{n}\lvert S_j\cap\La\rvert\]
remains invariant, as well as condition (\textbf{C1}).

With \eqref{C2strong}, we can verify
\begin{equation}\label{C2}
 (S_{j,m}-S_{k,u})\cap q_d\La^{d-1}=\vn
\end{equation}
for $(j,m)\neq(k,u)$. If $j=k$, we have
\begin{eqnarray*}
(S_{j,m}-S_{j,u})\cap q_d\La^{d-1} &=& (S_j[m]-S_j[u])\cap(q_d\La^{d-1}+(m-u)e^d)\\
&\ssq& \mathfrak{D}S_j\cap q_d(\La\sm\La^{d-1})=\vn
\end{eqnarray*}
by (\textbf{C1}) for $S_j$ and $q_d$ and the fact that $m\equiv u\pmod{q_d}$, $m\neq u$.
Assume next that $j\neq k$. then
\begin{eqnarray*}
 (S_{j,m}-S_{k,u})\cap q_d\La^{d-1} &=& (S_j[m]-S_k[u])\cap(q_d\La^{d-1}+(m-u)e^d)\\
&\ssq& (S_j-S_k)\cap q_d\La=\vn
\end{eqnarray*}
by \eqref{C2strong} for $t=q_d$. Therefore, by induction \eqref{lowerdim} holds, thus
\begin{eqnarray*}
 \sum_{j=1}^{n}\lvert S_j\cap\La\rvert &=& \sum_{r=1}^{q_d}\sum_{m\equiv r\bmod{q_d}}\sum_{j=1}^{n}\lvert S_{j,m}\cap\La^{d-1}\rvert\\
&\leqslant& \prod_{i=1}^{d}q_d
\end{eqnarray*}
as desired.
\end{proof}

\section{Proof of Theorem \ref{mainthm}}

We recall that the successive minima of $K$ with respect to a lattice $\La$ are those of $\hDK$. By 
definition of the successive minima $\la_i(K,\La)$, there are $d$ linearly independent lattice vectors $\mathbf{a}^i$, 
$1\leqslant i\leqslant d$ such that
\[\mathbf{a}^i\in \tfrac{\la_i(K,\La)}{2}\DK\cap\La.\]
Then we construct a basis of $\La$, say $\mathbf{e}^1,\dotsc,\mathbf{e}^d$, such that
\[\lin(\mathbf{a}^1,\dotsc,\mathbf{a}^i)=\lin(\mathbf{e}^1,\dotsc,\mathbf{e}^i)\]
for all $i$, $1\leqslant i\leqslant d$. Furthermore, we define the following subgroups of $\La$:
\[\La^i:=\mathbb{Z}\mathbf{e}^1\oplus\dotsb\oplus\mathbb{Z}\mathbf{e}^i\]
It is true that
\begin{equation}\label{mprop}
\interior(\tfrac{\la_i}{2}\DK)\cap\La\ssq\La^{i-1}.
\end{equation}
Indeed, let $j$ be the maximal index such that $\mathbf{a}^j\in\interior(\tfrac{\la_i}{2}\DK)$, or equivalently, $\la_{j+1}=\la_i$.
Then, by the definition of $\la_{j+1}$, there must be exactly $j$ linearly independent vectors in the above intersection. A choice
of such vectors is $\mathbf{a}^1,\dotsc,\mathbf{a}^j$, hence
\begin{equation*}
\interior(\tfrac{\la_i}{2}\DK)\cap\La\ssq\La^{j}\ssq\La^{i-1}.
\end{equation*}
The following Lemma was proven in \cite{M11}; we provide it here for completion.

\begin{lemma}\label{lem1}
Let $K\ssq\RR^d$ be a convex body and $\La$ be a lattice. For each real $n_i$, satisfying $n_i>2/\lambda_i$, we have
\[\DK\cap n_i(\La\sm \La^{i-1})=\vn.\]
In particular,
\[\interior(\DK)\cap \frac{2}{\la_i}(\La\sm\La^{i-1})=\vn.\]
\end{lemma}

\begin{proof}
Assume otherwise; then the intersection
\[\tfrac{1}{n_i}\mathfrak{D}K\cap(\Lambda\setminus\Lambda^{i-1})\]
would be nonempty. The left part of this intersection is a subset of
\[\interior(\tfrac{\lambda_i}{2}\mathfrak{D}K),\]
since $n_i>2/\lambda_i$. Therefore, the intersection
\[\interior(\tfrac{\lambda_i}{2}\mathfrak{D}K)\cap(\Lambda\setminus\Lambda^{i-1})\]
is nonempty, contradicting (\ref{mprop}) above, as was to be shown.
\end{proof}

Now consider an ellipsoid $E\ssq\RR^d$ and a lattice $\La$. There is some $A\in\GL_d(\RR)$ such that $AE$ is a sphere; we have
$G(AE,A\La)=G(E,\La)$, as well as $\la_i(AE,A\La)=\la_i(E,\La)$, for all $i$, $1\leqslant i\leqslant d$, therefore $q_i$ 
remain also invariant under the action of $\GL_d(\RR)$. So, without loss of generality, we may assume that $E=S$ is a 
sphere.

Let $\mathbf{e}_1,\dotsc,\mathbf{e}_d$ an appropriate basis of $\La$, defined as above. Now we apply Theorem \ref{strong}
 for $n=1$, $S_1=S$, the lattice $\La$ with basis $\mathbf{e}_1,\dotsc,\mathbf{e}_d$, $q_i=q_i(S,\La)$ for $1\leqslant i
\leqslant d$ and $q_{d+1}=1$. Condition (\textbf{C1}) holds by Lemma \ref{lem1} and condition (\textbf{C2}) holds vacuously
when $n=1$. Thus,
\[\lvert S\cap\La\rvert\leqslant \prod_{i=1}^{d}q_d,\]
completing the proof.

\end{document}